\documentclass{amsart} 

\usepackage[paperheight=11in, 
    paperwidth=8.5in,
    outer=1.25in,
    inner=1.25in,
    bottom=1.4in,
    top=1.4in]{geometry} 

\usepackage{amsmath, amsxtra, amsthm, amsfonts, amssymb}
\usepackage{mathrsfs}
\usepackage{graphicx}
\usepackage{url}
\usepackage{color}
\usepackage{bbm}
\usepackage{tikz-cd}

\usepackage[T1]{fontenc}
\usepackage{enumitem}

 \usepackage[hidelinks,backref=page]{hyperref} 
 \hypersetup{
    colorlinks,
    citecolor=magenta,
    filecolor=magenta,
    linkcolor=blue,
    urlcolor=black
}

\usepackage[capitalize]{cleveref}

\theoremstyle{definition}
\newtheorem{thm}{Theorem}[section]
\newtheorem{cor}[thm]{Corollary}
\newtheorem{lem}[thm]{Lemma}
\newtheorem{prop}[thm]{Proposition}
\newtheorem{defn}[thm]{Definition}
\newtheorem{definition}[thm]{Definition}

\newtheorem{rem}[thm]{Remark}

\newcommand{\RR}{\mathbb R}
\newcommand{\CC}{\mathbb C}
\newcommand{\PP}{\mathbb P}

\newcommand{\cM}{\mathcal{M}}
\newcommand{\cA}{\mathcal{A}}
\newcommand{\bA}{\overline \cA}
\newcommand{\rM}{\mathrm{M}}
\newcommand{\tM}{\widetilde{\cM}}
\newcommand{\A}{\mathfrak{A}}

\def\bOmega{{\overline \Omega}}
\def\tOmega{\Omega}

\def\T{{\mathcal{T}}}

\def\Res{\operatorname{Res}}
\def\L{{\mathcal L}}
\def\triang{{\Delta}}
\def\tV{\tilde V}
\def\pt{{\rm pt}}
\def\dlog{\operatorname{dlog}}
\newcommand\OS{{\operatorname{OS}}}
\newcommand\oOS{{\overline{\OS}}}
\newcommand{\be}{{\overline e}}

\title{Canonical forms of oriented matroids}
\author{Christopher Eur}
\author{Thomas Lam}

\address{Department of Mathematics, University of Michigan, 2074 East Hall, 530 Church Street, Ann Arbor, MI 48109-1043, USA}
\email{\href{mailto:tfylam@umich.edu}{tfylam@umich.edu}}

\address{Carnegie Mellon University. Pittsburgh, PA. USA}
\email{ceur@cmu.edu}

\begin{document}
\begin{abstract}
Positive geometries are semialgebraic sets equipped with a canonical differential form whose residues mirror the boundary structure of the geometry.  Every full-dimensional projective polytope is a positive geometry.  Motivated by the canonical forms of polytopes, we construct a canonical form for any tope of an oriented matroid, inside the Orlik--Solomon algebra of the underlying matroid.  Using these canonical forms, we construct bases for the Orlik--Solomon algebra of a matroid, and for the Aomoto cohomology.  These bases of canonical forms are a foundational input in the theory of matroid amplitudes introduced by the second author.
\end{abstract}

\maketitle

\vspace{-20pt}

\section{Introduction}\label{sec:intro}

The emerging field of positive geometries bridges combinatorial algebraic geometry in mathematics and the study of scattering amplitudes in physics.
Functions such as scattering amplitudes or cosmological correlators are determined by their behavior at singularities, and positive geometries encode the boundary structures induced by these singularities via their canonical forms.
See \S\ref{ssec:defposgeom} for a definition of positive geometries.  Examples of positive geometries include projective polytopes and positive parts of toric varieties \cite{ABL}, totally nonnegative Grassmannians and flag varieties \cite{LamPG}, wondertopes \cite{BEPV}, and conjecturally, the amplituhedron \cite{AT}.

\smallskip
Let $\bA$ be a collection of hyperplanes in a real projective space $\PP^d(\RR)$.
The connected components of the complement $\overline U(\RR) = \PP^d(\RR)\setminus \bigcup \bA$ of the union of the hyperplanes are called \emph{chambers}.
When $\bA$ is \emph{essential}, i.e.\ the intersection of the hyperplanes is empty, the closure of each chamber is a projective polytope, which is a positive geometry.
Thus, an essential hyperplane arrangement $\bA$ is equipped with a collection of canonical forms $\Omega_{\overline P}$, one for each closed chamber $\overline P$.  These canonical forms, together with the classes they represent in algebraic de Rham cohomology of the complexification $\overline U(\CC)$, are the starting point of this work.

\smallskip
Oriented matroids are combinatorial abstractions of real hyperplane arrangements.  We point to \cite{ZAK24} for a survey.  Many outstanding problems about oriented matroids concern whether a property of real hyperplane arrangements extends to all oriented matroids.
Here, we show that the notion of canonical forms of chambers of real hyperplane arrangements extends well to all oriented matroids.

\smallskip
In Section~\ref{sec:mainthm}, we define the canonical form $\bOmega(\cM)$ of an oriented matroid $\cM$ as an element of the reduced Orlik--Solomon algebra $\oOS(\rM)$ of the underlying matroid $\rM$ of $\cM$.
We show in Theorem~\ref{thm:main} that $\bOmega(\cM)$ is the unique solution to a recursion, defined using the residue sequence (\cref{prop:OSexact}) of Orlik--Solomon algebras.
This recursion mirrors the defining recursive property of positive geometries.
In \cref{def:main}, we extend the definition to the canonical form $\bOmega(P)$ of any tope $P$ of $\cM$.

\smallskip
In \cref{sec:real}, we explain how our synthetic canonical form $\bOmega(P)$ equals the usual canonical form $\Omega_{\overline P}$ of a projective polytope $\overline P$ when the oriented matroid $\cM$ is realizable as a hyperplane arrangement.
In the realizable setting, Yoshinaga \cite{Yoshinaga} studied the canonical forms of chambers in a different formulation.  While this manuscript was in preparation, we learnt of the work of Brown and Dupont \cite{BD} who studied the canonical forms of chambers and were, like us, motivated by positive geometry.

\smallskip
The combinatorics of the Orlik--Solomon algebra $\oOS^\bullet(\rM)$ is traditionally studied using the no-broken-circuit (NBC) complex.
In \cref{thm:basis}, we produce a basis for $\oOS^\bullet(\rM)$ consisting of canonical forms, 
which provides a combinatorial approach to $\oOS^\bullet(\rM)$ that serves as an alternative to the NBC complex.
For instance, we recover the result of \cite{GZ} that expressed the graded dimensions of $\oOS^\bullet(\rM)$ in terms of the number of topes, and in the realizable case, we obtain Yoshinaga's chamber basis \cite{Yoshinaga}.  It would be interesting to give a simple description of the multiplicative structure constants of $\oOS^{\bullet}(\rM)$ with respect to the canonical form basis (see \cref{rem:mult}).

\smallskip
In \cref{thm:Aomoto}, we show that the canonical forms of bounded topes is a basis for the Aomoto cohomology $H^{r-1}(\oOS_\CC^\bullet, \omega)$ when the parameters in $\omega$ are generic.  In the realizable case, under a genericity hypothesis, Aomoto cohomology is isomorphic to a twisted cohomology group $H^\bullet(\overline U(\CC),\L_\omega)$ with values in a local system $\L_\omega$.  This is the starting foundational step for \cite{Lamampl}, where the result of this work is applied to define amplitudes for matroids.

\smallskip
Finally, in \cref{sec:examples} we discuss some examples of the construction of this work.

\subsection*{Acknowledgements}
C.E. was partially supported by the National Science Foundation grant DMS-2246518.
T.L. acknowledges support from the National Science Foundation under grants DMS-1953852 and DMS-2348799, and from the Simons Foundation under a Simons Fellowship.  We thank the Institute for Advanced Study, Princeton for supporting us during visits during which parts of this manuscript was completed.  T.L. thanks the Sydney Mathematical Research Institute for supporting a research visit during which parts of this manuscript was completed.  We thank the referees for their careful reading and helpful suggestions.

\section{Canonical forms of oriented matroids and topes}\label{sec:mainthm}

Let $E$ be a finite set.
For a sign vector $X \in \{+,0,-\}^E$, we write $X^-$ for the set $\{i\in E : X(i) = -\}$.
For a sequence $S$, we often abuse notation to write $S$ also for the set of its entries.
For oriented matroids, we follow the terminology and conventions as given in \cite{BLVSWZ99}.
Throughout, let $\cM$ be a loopless oriented matroid $\cM$ of rank $r$ on a ground set $E$.
We say that $\cM$ is \emph{acyclic} if it has no positive circuit.  For a subset $S\subseteq E$, we say that $S$ is \emph{acyclic in $\cM$} if the restriction $\cM|_S$ is acyclic.
Let $\rM$ denote the underlying matroid of $\cM$, and let $\A(\rM)$ denote the set of its atoms (i.e.\ rank 1 flats).
We write simply $\A$ when the matroid $\rM$ is understood.  A subset $I \subseteq \A$ is \emph{dependent} if its rank in $\rM$ is less than $|I|$.

\medskip
The \emph{Orlik--Solomon algebra} $\OS^\bullet_R(\rM)$ of $\rM$ over a unital commutative ring $R$ is the graded skew-commutative ring defined as a quotient of the exterior algebra of $R^{\A}$ by
\[
\OS^\bullet_R(\rM) := \frac{\bigwedge^\bullet R^{\A}}{\langle \partial e_S : S\subseteq \A \text{ is dependent}\rangle}.
\]
Here, we have denoted $e_S := e_{s_1} \wedge \dotsm \wedge e_{s_k}$ for a sequence $S = (s_1, \dotsc, s_k)$ in $\A$, where $e_a$ denotes the standard basis vector of $a\in \mathfrak A$ in $R^\A$, and we denote
\begin{equation}\label{eq:partialdefn}
\partial e_S := \sum_{i = 0}^{k-1} (-1)^{i} e_{S\setminus s_{k-i}}.
\end{equation}
We caution that this convention differs from the one often found in the literature; see Remark~\ref{rem:convention}.
The \emph{reduced Orlik--Solomon algebra} of $\rM$ is the subalgebra $\oOS^\bullet_R(\rM) := \ker \big(\partial: \OS^\bullet_R(\rM) \to \OS^\bullet_R(\rM)\big)$.
We omit the subscript when the ring $R$ is understood or irrelevant.  For a sequence $I = (i_1, \dotsc, i_k)$ in $E$, we abuse notation to write $e_I$ for $e_{\overline I}$ where $\overline I = (\operatorname{closure}_\rM(i_1), \dotsc, \operatorname{closure}_\rM(i_k))$.
We collect some facts about Orlik--Solomon algebras that we will need.

\begin{prop}\label{prop:partialiso}{\cite[Proposition 3.2]{Dimca}}
We have that $\oOS_R^\bullet(\rM) = \partial (\OS_R^\bullet(\rM))$, and the map $\partial: \OS_R^r(\rM) \longrightarrow \oOS_R^{r-1}(\rM)$ is an isomorphism of $R$-modules.
\end{prop}

For an atom $a \in \A$, we write $\rM \setminus a$ for the deletion of $a$, and $\rM/a$ for the contraction by $a$.
\begin{prop}\label{prop:OSexact}
For every atom $a\in \A$, we have a short exact sequence of $R$-modules
\[
0\to \OS^\bullet_R(\rM\setminus a) \overset{\iota_a}\to \OS^\bullet_R(\rM) \overset{\Res_a}\to \OS^{\bullet-1}_R(\rM/a) \to 0
\]
where $\iota_a(e_I) = e_I$, and
\[
\Res_a(e_I) := \begin{cases}
e_{I\setminus a} & \text{if $I = (i_1, \dots, i_{k-1}, a)$, and}\\
0 & \text{if $I \cap a = \emptyset$}.
\end{cases}
\]
These maps restrict to give the short exact sequence
\begin{equation}\label{eq:oRes}
0\to \oOS^\bullet_R(\rM\setminus a) \to \oOS^\bullet_R(\rM) \to \oOS^{\bullet-1}_R(\rM/a) \to 0.
\end{equation}
\end{prop}

\begin{proof}
Fixing an element $q\in \A$ with $q\neq a$, the \emph{affine Orlik--Solomon algebra} $\OS^\bullet_R(\rM,q)$ is defined by
\begin{equation}\label{eq:affine}
\OS^\bullet_R(\rM,q) =  \frac{\bigwedge^\bullet R^{\A\setminus q}}{\langle \epsilon_S : q \subseteq \operatorname{closure}_\rM(S)\rangle + \langle \partial \epsilon_S : \text{$q\not\subseteq \operatorname{closure_\rM(S)}$ and $S$ is dependent}\rangle}
\end{equation}
where the $\epsilon_i$ denotes the standard basis vectors of $R^{\A\setminus q}$.
\cite[Theorem 3.4]{Dimca} states that there is an isomorphism $\OS^\bullet(\rM,q) \overset\sim\to \oOS^\bullet(\rM)$ given by $\epsilon_I \mapsto \partial e_{Iq}$.  
\cite[Theorem 3.65]{OrlikTerao} states the appropriate short exact sequence for the affine Orlik--Solomon algebras, where description of the middle two maps in the short exact sequence is identical to the description in the proposition but with $e_I$ replaced by $\epsilon_I$.\footnote{The cited statements in \cite{Dimca} and \cite{OrlikTerao} are established for hyperplane arrangements, but the proof is valid for any matroid.}
Both short exact sequences now follow, and the one for $\oOS^\bullet$ is the restriction of the one for $\OS^\bullet$ since $\Res_a(\partial e_{I',a,q}) = \partial e_{I',q}$ and $\Res_a(\partial e_{Iq}) = 0$ if $I \cap a = \emptyset$.
\end{proof}

\begin{rem}\label{rem:convention}
We caution that our conventions for $\partial$ and $\Res$ differs from those in \cite{Dimca} and \cite{OrlikTerao}.
In those references, $\partial e_S$ and $\Res_a$ are defined as $\partial e_S = \sum_{i = 1}^k (-1)^{i+1}e_{S\setminus s_i}$ and $\Res_a(e_{aI}) = e_I$, which differs from our definition by a sign.
Our choice of conventions here agrees with those of \cite{ABL} where positive geometries were introduced.
\end{rem}

We record the following consequence for future use.

\begin{lem}\label{lem:Resinj}
The map $\prod_{a \in \A} \Res_a: \oOS^d(\rM) \to \prod_{a\in \A} \oOS^{d-1}(\rM/a)$ is an injection for $d\geq 1$.
\end{lem}

\begin{proof}
For distinct atoms $a\neq b\in \A$, the description of the maps in Proposition~\ref{prop:OSexact} implies that
\[
\begin{tikzcd}
\OS^\bullet(\rM\setminus a)  \ar[r, "\Res_b"] \ar[d,"\iota_a"'] &\OS^{\bullet-1}(\rM\setminus a / b) = \OS^{\bullet-1}(\rM / b \setminus a) \ar[d,"\iota_a"]\\
\OS^\bullet(\rM) \ar[r, "\Res_b"] & \OS^{\bullet-1}(\rM/b)
\end{tikzcd}
\]
commutes, i.e.\ we have $\Res_b \circ \iota_a = \iota_a \circ \Res_b$.  We now induct on $|\A|$ to prove the lemma.
The statement holds vacuously when $|\A| = 1$, since $\oOS^d(\rM) =  0$ for all $d\geq 1$ in that case.
Now, suppose $\omega \in \oOS^d(\rM)$ is in the kernel of the map $\prod_{a \in \A} \Res_a: \oOS^d(\rM) \to \prod_{a\in \A} \oOS^{d-1}(\rM/a)$.   For $a \in \A$, since $\Res_a(\omega) = 0$, by Proposition~\ref{prop:OSexact} we have $\omega = \iota_a(\omega')$ for some $\omega' \in  \oOS^\bullet(\rM \setminus a)$.
For every $b \in \A \setminus a$, which is nonempty when $|\A|>1$, we then have
\[
\iota_a(\Res_b(\omega')) = \Res_b(\iota_a(\omega')) = \Res_b(\omega) = 0,
\]
so that $\Res_b(\omega') = 0$ since $\iota_a$ is an injection.  By induction, we conclude $\omega' = 0$, and thus $\omega = 0$.
\end{proof}

To state the main theorem, we will use the notion of a \emph{triangulation} of an oriented matroid $\cM$, which is a collection of bases of $\cM$ satisfying certain conditions \cite{San02}.  The only properties of triangulations we need are collected in \cref{lem:triang}.
We fix one more notation as follows.
For a choice of a chirotope $\chi$ on $\cM$, an atom $a\in \A$, and an element $i\in a$ in $E$, the function
\[
(\chi/i)(B) := \chi(B,i) \quad\text{ for any ordered basis $B$ of $\rM/a$}
\]
is a chirotope on $\cM/a$.  When $a$ is acyclic in $\cM$, the chirotope $\chi/i$ is independent of the choice of $i\in a$, so we may write $\chi/a$ in that case.  We now state the main theorem.

\begin{thm}\label{thm:main}
For each pair $(\cM, \chi)$ of a loopless oriented matroid $\cM$ of rank $r$ and a chirotope $\chi$ on $\cM$, choose any triangulation $\triang$ of $\cM$ and define
\begin{equation}\label{eq:cantriang}
\bOmega(\cM, \chi) := \sum_{B\in \triang} \chi(B)\partial e_B.
\end{equation}
The assignment $(\cM, \chi) \mapsto \bOmega(\cM,\chi)$ is the unique way of assigning to each pair $(\cM, \chi)$ an element in $\oOS^{r-1}(\rM)$ satisfying the following properties:
\begin{enumerate}[label = (\arabic*)]
\item\label{main1} When the rank $r = 1$, if $\cM$ is not acyclic then $\bOmega(\cM,\chi) = 0$, and if $\cM$ is acyclic then $\bOmega(\cM,\chi) = (\chi/E)(\emptyset) \in \{\pm 1\}$ (or equivalently $\bOmega(\cM,\chi) = \chi(i)$ for any $i\in E$).
\item\label{main2} For every atom $a \in \A(\rM)$, we have
\[
\Res_a\bOmega(\cM,\chi) = \begin{cases}
\bOmega(\cM/a, \chi/a) & \text{if $a$ is acyclic in $\cM$}\\
0 &\text{otherwise}.
\end{cases}
\]
\end{enumerate}
We call $\bOmega(\cM,\chi)$ the \emph{canonical form} of the pair $(\cM,\chi)$.
\end{thm}

In particular, \cref{thm:main} implies that the expression $\sum_{B\in \triang} \chi(B)\partial e_B$, as an element in the reduced Orlik--Solomon algebra of $\rM$, is independent of the triangulation $\triang$ of $\cM$.
Note that by definition $\bOmega(\cM,-\chi) = -\bOmega(\cM,\chi)$.
To avoid nested inductions in the proof of the theorem, let us observe the following.

\begin{lem}\label{lem:acycliczero}
Suppose Theorem~\ref{thm:main} holds for all loopless oriented matroids of rank at most $s$.  Then, for $\cM$ a loopless oriented matroid of rank $r\leq s$ and $\chi$ a chirotope on $\cM$, the following holds:
\begin{quote}
If $\cM$ is not acyclic, then $\bOmega(\cM,\chi) = 0$, that is, $\sum_{B\in \triang} \chi(B)\partial e_B = 0$ for any triangulation $\triang$ of $\cM$.
\end{quote}
\end{lem}

\begin{proof}
We induct on the rank $r$, where the base case $r = 1$ is clear.  Suppose now $\cM$ has rank $1<r\leq s$ and is not acyclic.  We first note that if an atom $a\in \A$ is acyclic in $\cM$, then $\cM/a$ is not acyclic.
Indeed, in this case a positive circuit of $\cM$ cannot be contained in $a$, so its image in $\cM/a$ is a positive vector, which contains a positive circuit of $\cM/a$.  Thus, property \ref{main2} in Theorem~\ref{thm:main}, the induction hypothesis, and Lemma~\ref{lem:Resinj} together imply that $\bOmega(\cM,\chi) = 0$.
\end{proof}

\begin{proof}[Proof of Theorem~\ref{thm:main}]

We first show uniqueness by induction on the rank $r$.
The base case $r = 1$ holds by definition.  For $\cM$ of rank $r>1$, Lemma~\ref{lem:Resinj} implies that $\bOmega(\cM,\chi)$ is determined by property \ref{main2} and the induction hypothesis.

For the existence, note first that for a fixed triangulation $\triang$, the formula
\[
\bOmega(\cM,\chi) := \sum_{B\in \triang} \chi(B) \partial e_B
\]
is well-defined since for each basis $B$ in the triangulation $\triang$, the element $\chi(B)\partial e_B$ is independent of the ordering of $B$.
We show by induction on the rank $r$ that $\bOmega(\cM,\chi)$ satisfies the properties \ref{main1} and \ref{main2}.
The base case $r = 1$ follows from \cite[first part of Theorem 2.4(e)]{San02}, which states the following.  When $r = 1$, a triangulation of an acyclic $\cM$ consists of one basis, and a triangulation of a non-acyclic $\cM$ consists of two bases $i,j \subseteq E$ such that $\chi(i) = -\chi(j)$.
This verifies property \ref{main1}.

Now suppose $\cM$ has rank $r>1$.  For all $\cM'$ of rank at most $r-1$, the induction hypothesis implies that the formula $\sum_{B\in \triang'} \chi(B) \partial e_B$ is independent of the triangulation $\triang'$ of $\cM'$, and that $\bOmega_{\triang'}(\cM',\chi') = 0$ if $\cM'$ is not acyclic by Lemma~\ref{lem:acycliczero}.
\cite[Theorem 2.4.(e)]{San02} states that for any element $i\in E$, the collection $\triang_i := \{B\setminus i: B\in \triang \text{ and } i\in B\}$ is either empty or a triangulation of $\cM/a$, where $a\in \A$ is the atom containing $i$.
For $a\in \A$, we compute that 
\begin{equation}\label{eqn:res}\tag{$\dagger$}
\Res_a \Big(\sum_{B\in \triang} \chi(B)\partial e_B \Big) = \sum_{\substack{B\in \triang\\ B \cap a \neq \emptyset}} \chi(B) \partial e_{B\setminus a} = \sum_{\substack{i\in a\\ \triang_i \neq \emptyset}} \Big( \sum_{B' \in \triang_i} (\chi/i)(B') \cdot \partial e_{B'}\Big).
\end{equation}
It follows from Lemma~\ref{lem:triang} below that exactly one of the following three cases occur:
\begin{enumerate}[label = (\roman*)]
\item The collection $\triang_i$ is empty for every $i\in a$, and in this case $\cM/a$ is not acyclic.
\item The collection $\triang_i$ is nonempty for a unique $i\in a$, and in this case $a$ is acyclic in $\cM$.
\item The collection $\triang_i$ is nonempty for exactly two $i\in a$, and in this case the two such elements $i_1$ and $i_2$ form a positive circuit in $a$, so that $a$ is not acyclic.
\end{enumerate}
In the case (i), from \eqref{eqn:res} we have $\Res_a \bOmega(\cM,\chi) = 0$, which equals $\bOmega(\cM/a,\chi/a) = 0$ by the induction hypothesis since $\cM/a$ is not acyclic.
In the case (ii), from \eqref{eqn:res} we have $\Res_a \bOmega(\cM,\chi) = \sum_{B'\in \triang_i} (\chi/i)(B') \cdot \partial e_{B'}$, which equals $\bOmega(\cM/a,\chi/a)$ by the induction hypothesis that the expression is independent of the triangulation for $\cM/a$.
Lastly, in the case (iii), from \eqref{eqn:res} we have
\begin{align*}
\Res_a \bOmega(\cM,\chi) &= \sum_{B'\in \triang_{i_1}} (\chi/i_1)(B') \cdot \partial e_{B'} + \sum_{B'\in \triang_{i_2}} (\chi/i_2)(B') \cdot \partial e_{B'} \\
&= \bOmega(\cM/a, \chi/i_1) + \bOmega(\cM/a, \chi/i_2) \\
&=  \bOmega(\cM/a, \chi/i_1) + \bOmega(\cM/a, -\chi/i_1)  = 0,
\end{align*}
where the second equality follows from the induction hypothesis, and the third equality from that $\{i,j\}$ is a positive circuit.
In every case, we find that property \ref{main2} is satisfied.
\end{proof}

\begin{lem}\label{lem:triang}
Let $\triang$ be a triangulation of a loopless oriented matroid $\cM$, and $a\in \A(\rM)$ be an atom.
\begin{enumerate}
\item If two bases $B,B'\in \triang$ satisfy $B\cap a = \{i\}$ and $B'\cap a = \{j\}$ for some $i,j\in E$ where $\{i,j\}$ is acyclic in $\cM$, then $i = j$.
\item If $\cM/a$ is acyclic, then there exists $i\in a$ such that $B\cap a = \{i\}$ for some $B\in \triang$.
\item If $B\cap a \neq \emptyset$ for some $B\in \triang$ and $a$ is not acyclic, then there exists $B' \in \triang$ such that $B'\cap a \neq \emptyset$ and $(B\cap a) \cup (B' \cap a)$ is not acyclic in $\cM$.
\end{enumerate}
\end{lem}

\begin{proof}
For (1), since $i$ and $j$ are in the same atom $a$ and $\{i,j\}$ is acyclic, the extension $\cM\cup p$ by a copy $p$ of $i$ satisfies $p\in \operatorname{conv}_{\cM\cup p}(B) \cap \operatorname{conv}_{\cM\cup p}(B')$.
The proper intersection property of triangulations \cite[Theorem 2.4(a)]{San02} then implies that $i\in \operatorname{conv}_{\cM\cup p}(B\cap B')$.  Hence, if $i\neq j$, then the span of $B\cap B'$ does not contain the atom $a$ of $i$ and $j$, which contradicts that $p\in \operatorname{conv}_{\cM\cup p}(B\cap B')$.

For (2), choose any $i\in a$ and let $\cM\cup p$ be the extension by a copy $p$ of $i$.
\cite[Theorem 2.4(b)]{San02} states that $p\in \operatorname{conv}_{\cM\cup p}(B)$ for some $B\in \triang$.  We claim that the signed circuit $C$ in $B\cup p$ witnessing $p\in \operatorname{conv}_{\cM\cup p}(B)$ contains only two elements $\{p,j\}$, in which case $j$ is necessarily in $a$, so that $B\cap a = \{j\}$.  Indeed, otherwise $C\setminus p$ is not a loop, and hence a positive vector in $\cM/a$, contradicting that $\cM/a$ is acyclic.

For (3), repeated application of \cite[Theorem 2.4(e)]{San02} implies that the contraction of $\cM$ by $B\setminus a$ yields a triangulation of the restriction $\cM|a$, which necessarily has a face containing an element that forms a positive circuit with $B\cap a$.
\end{proof}

\begin{rem}
We do not know how to prove that the formula \eqref{eq:cantriang} is independent of triangulation directly without using the recursive characterization of the canonical form.  We remark that the flip-graph of triangulations of an oriented matroid is not necessarily connected, even for realizable oriented matroids \cite{San00}.
\end{rem}

We now define the canonical form of a tope in an oriented matroid.

\begin{defn}\label{def:main} For a tope $P$ of $\cM$ and a chirotope $\chi$ of $\cM$, let $_P\cM$ be the reorientation of $\cM$ whose chirotope is $_P\chi(B) := (-1)^{|B \cap P^-|}\chi(B)$.  Define the \emph{canonical form of $P$} to be
\[
\bOmega(P,\chi) := \bOmega( _{P}\cM, _{P}\chi) \in \oOS^{r-1}(\rM).
\]
\end{defn}

Note that the tope $P$ becomes the tope $+^E$ in the acyclic reoriented matroid $_P\cM$.

\medskip
We now set up some notations to state the recursive property of the canonical form of a tope.
For an atom $a\in \mathfrak A$ of $\cM$, we say that $a$ is a \emph{facet} of a tope $P$, denoted $a\in \partial P$, if $P|_{E\setminus a}$ is a tope of $\cM/a$, or equivalently, if $P_a$ defined by
\begin{equation}\label{eq:Pa}
P_a(i) = \begin{cases} 
P(i) &\text{if $i\notin a$ and}\\
0 &\text{if $i\in a$}
\end{cases}
\end{equation}
is a covector of $\cM$.
When $a\in \partial P$, we write $P/a$ for $P|_{E\setminus a}$.  We also define $\chi/^P a$, a chirotope of $\cM/a$ by
\[
(\chi/^P a)(I) := P(i)\chi(I,i) \quad\text{for a basis $I$ of $\cM/a$ and any $i\in a$}.
\]
This is independent of the chosen $i\in a$ since $P(i) = P(i')$ and $\chi(I,i) = \chi(I,i')$ if $i,i'\in a$ are parallel, and $P(i) = -P(i')$ and $\chi(I,i) = - \chi(I,i')$ if $i,i'\in a$ are anti-parallel.
Note that $_{P/a}\big((_P\chi)/a\big)$, the reorientation by $P/a$ of the contraction by $a$ of the reorientation by $P$ of $\chi$ (where $(_P\chi)/a$ is well-defined since $a$ is acyclic in $_P\cM$), is equal to $\chi/^P a$.
Canonical forms $\bOmega(P,\chi)$ of topes have the following recursive characterizing property.

\begin{cor}\label{cor:topeform}
The canonical forms $\bOmega(P,\chi)$ satisfy the characterizing recursion:
\begin{enumerate}[label=(\roman*)]
\item If $\cM$ has rank $r=1$, then $\bOmega(P,\chi) = P(e)\chi(e) 1 \in \oOS^0(\rM)$ for any $e\in E$.
\item
If $\cM$ has rank $r > 1$ then for any $a \in \mathfrak A$ we have
$$
\Res_a \bOmega(P,\chi) = \begin{cases}  \bOmega(P/a, \chi/^P a) & \text{if $a \in \partial P$}, \\
0 & \text{if $a \notin \partial P$.}
\end{cases}
$$
\end{enumerate}
\end{cor}

\begin{proof}
When $\cM$ is acyclic, for the tope $P = +^E$, we have $a\in \partial P$ if and only if $\cM/a$ is acyclic.
Since reorienting commutes with contracting for oriented matroids, and since $_{P/a}\big((_P\chi)/a\big) = \chi/^P a$, the desired corollary now follows from Theorem~\ref{thm:main} and Lemma~\ref{lem:acycliczero}.
\end{proof}

We now discuss a variant of $\bOmega(\cM,\chi)$ for the non-reduced Orlik--Solomon algebra.
Let us define $\tOmega(\cM,\chi) \in \OS^r(\rM)$ by replacing $\partial e_B$ by $e_B$ in \eqref{eq:cantriang}, and similarly define $\tOmega(P,\chi)$.

\begin{thm}
The canonical form $\tOmega(P,\chi) \in \OS^r(\rM)$ does not depend on choice of triangulation of ${}_P \cM$, and satisfies the characterizing recursion
\begin{enumerate}[label=(\roman*)]
\item If $\cM$ has rank $r=0$, then $\tOmega(P,\chi) = \chi(\emptyset) 1 \in \OS^0(\rM)$.
\item
If $\cM$ has rank $r > 0$ then for any $a \in \mathfrak A$ we have
$$
\Res_a \tOmega(P,\chi) = \begin{cases}  \tOmega(P/a, \chi/^P a) & \text{if $a \in \partial P$}, \\
0 & \text{if $a \notin \partial P$.}
\end{cases}
$$
\end{enumerate}
\end{thm}

\begin{proof}
Consider the extension $\cM \oplus \star$ of $\cM$ by a coloop $\star$ with the chirotope $\chi'$ defined by $\chi'(\star,B) = \chi(B)$.
Since $\star$ is never in the closure of a set of other atoms, we find that the affine Orlik--Solomon algebra $\OS^\bullet(\rM\oplus \star)$ is equal to $\OS^\bullet(\rM)$ by \eqref{eq:affine}, so that we conclude $\oOS^\bullet(\rM\oplus \star) \simeq \OS^\bullet(\rM)$ via $\partial e_{I,\star} \mapsto e_I$.  The theorem now follows by applying \cref{cor:topeform} to $\cM \oplus \star$, since a tope $P$ of $\cM$ defines a tope $(P,+)$ of $\cM\oplus \star$ and a triangulation of $_P\cM$ lifts to that of $_P\cM \oplus \star$.

Alternatively, since $\partial \tOmega(P,\chi) = \bOmega(P,\chi)$, one may use \cref{prop:partialiso} and repeat the arguments for \cref{thm:main} and \cref{cor:topeform} to conclude the theorem.
\end{proof}

\section{Realizable case}\label{sec:real}

\subsection{Definition of positive geometry}\label{ssec:defposgeom}
Let us recall the definition of positive geometries, along with conventions for residues and orientations following \cite{ABL}.

\medskip
\def\Int{{\rm Int}}
Let $X$ be a smooth complex $d$-dimensional irreducible algebraic variety, $\omega$ a meromorphic $d$-form on $X$, and $H \subset X$ an (irreducible) hypersurface on $X$.\footnote{Positive geometries are defined more generally for normal varieties, but we will only need the case when the varieties involved are smooth.}  Assume that $\omega$ has at most simple poles on $H$.  Then, the \emph{residue} $\Res_H \omega$ is the $(d-1)$-form on $H$ defined as follows.  Let $f$ be a local coordinate such that $f$ vanishes to order one on $H$.  Write 
$$
\omega = \eta \wedge \frac{df}{f} + \eta' 
$$
for a $(d-1)$-form $\eta$ and a $d$-form $\eta'$, both without poles along $H$.  Then the restriction
\begin{equation}\label{eq:resdef}
\Res_H \omega:= \eta|_H
\end{equation}
is a well-defined meromorphic $(d-1)$-form on $H$, called the \emph{residue} of $\omega$, and does not depend on the choices of $f, \eta,\eta'$.

\medskip
We now further assume that $X$ is a variety defined over $\RR$.
We equip the real points $X(\RR)$ with the analytic topology.
Let $X_{\geq 0} \subset X(\RR)$ be a closed semialgebraic subset such that the interior $X_{>0} = \Int(X_{\geq 0})$ is an oriented real $d$-manifold, and the closure of $X_{>0}$ recovers $X_{\geq 0}$.
Let $\partial X_{\geq 0}$ denote the boundary $X_{\geq 0} \setminus X_{ > 0}$ and let $\partial X$ denote the Zariski closure of $\partial X_{\geq 0}$.  Let $C_1,C_2,\ldots,C_r$ be the codimension 1 irreducible components of $\partial X$.  
Let $C_{i, \geq 0}$ denote the analytic closures of the interior of $C_i \cap \partial X_{\geq 0}$ in $C_i(\RR)$.
The spaces $C_{1,\geq 0}, C_{2,\geq 0},\ldots, C_{r,\geq 0}$ are called the boundary components, or facets of $X_{\geq 0}$.
When the interior $C_{i,>0} = \Int(C_{i,\geq 0})$ is orientable, its orientation is inherited from that of $X_{>0}$ as in the following remark.

\begin{rem}\label{rem:orientboundary}
For an oriented real manifold $M$ with an orientable boundary $N$ of codimension 1, we orient $N$ such that a basis $(v_1, \dotsc, v_{d-1})$ of tangent vectors on $N$ is positively oriented if and only if $(v_1, \dotsc, v_{d-1}, u)$ is positively oriented on $M$ for any inner normal vector $u$ of the boundary $N$ in $M$.
\end{rem}

\begin{definition}\label{def:PG}
  We call $(X,X_{\geq 0})$ a \emph{positive geometry} if there exists a unique nonzero rational $d$-form $\Omega(X,X_{\geq 0})$, called the \emph{canonical form}, satisfying the recursive axioms:
  \begin{enumerate}
    \item If $d = 0$, then $X = X_{\geq 0}= \pt$ is a point and we define $\Omega(X,X_{\geq 0})=\pm 1$ depending on the orientation.
    \item If $d>0$, then we require that $\Omega(X,X_{\geq 0})$ has poles only along the boundary components $C_i$, these poles are simple, and for each $i =1,2,\ldots,r$, we have
    \begin{equation}\label{eq:PGdef}
    \Res_{C_i}\Omega(X,X_{\geq 0})=\Omega(C_i,C_{i,\geq0}).
    \end{equation}
  \end{enumerate}
\end{definition}

\begin{thm}\cite{ABL}
Any $d$-dimensional polytope $\overline P \subset \RR^d \subset \PP^d(\RR)$ is a positive geometry.
\end{thm}
We denote the canonical form of a $d$-dimensional projective polytope $\overline P$ by $\Omega_{\overline P}$.
In the next subsection, we explain how Theorem~\ref{thm:main} recovers and sharpens the above theorem.

\subsection{Hyperplane arrangements}
Let $V$ be a $r$-dimensional real vector space, and let $\mathscr L = \{\ell_i \in V^\vee: i \in E\}$ be a set of linear functionals indexed by a finite set $E$.
We assume that $\mathscr L$ spans $V^\vee$
Let $\cM$ be the oriented matroid whose set of covectors is $\{(\operatorname{sign}(\ell_i(x)))_{i\in E} \in \{+,0,-\}^E \mid x\in V\}$.
Writing $H_i$ for the hyperplane $\{\ell_i = 0\} \subset V$, let $\cA = \{H_i \mid i\in E\}$ be the essential central hyperplane arrangement of $\mathscr L$, and denote by $\bA = \{\PP H_i \mid i \in E\}$ the essential projective hyperplane arrangement in $\PP V$.
We let $U  := V \setminus \bigcup \cA$ and $\overline U :=\PP V \setminus \bigcup \bA $ denote the respective arrangement complements.
To each tope $P$ of $\cM$, we associate the subset
\[
\mathring\sigma_P = \{v\in V : \operatorname{sign}(\ell_i(v)) = P(i)\},
\]
which is a chamber of $U(\RR)$, and whose closure $\sigma_P$ is a polyhedral cone.
In this way, we identify chambers of $U(\RR)$ (resp.\ $\overline U(\RR)$) with the topes of $\cM$ (resp.\ equivalences classes of topes where $P \sim -P$).

We assume that $\{\ell_i,\ell_j\}$ is linearly independent for any $i\neq j\in E$.
That is, the oriented matroid $\cM$ is simple of rank $r$, and the arrangement $\bA$ is essential.
Under the essential hypothesis, for a tope $P$ of $\cM$, the closure of the corresponding chamber in $\overline U(\RR)$ is a full-dimensional projective polytope, denoted $\overline P$, and thus has a canonical form $\Omega_{\overline P}$ once the interior of $\overline P$ is oriented.
We record our conventions for orientations in the following remark.

\begin{rem}\label{rem:orientP}
A choice of a chirotope $\chi$ for $\cM$ orients $V^\vee$ in the obvious way, which orients $V$ such that a basis is positively oriented if and only if its dual basis is.
The projective space $\PP V$ is not orientable if $r\geq 3$ and odd, but given an orientation of $V$ and an additional linear functional $q\in V^\vee$, we orient $\PP V \setminus \{q = 0\}$ as follows.
We orient the subset $\{v\in V : q(v) \geq 1\}$ as a full-dimensional submanifold of $V$ with boundary.  Then, we identify $\PP V\setminus \{q = 0\}$ with $\{v\in V : q(v) = 1\}$, and consider it the boundary of the subset $\{v\in V : q(v) \geq 1\}$, which induces an orientation on $\{v\in V : q(v) = 1\}$ via Remark~\ref{rem:orientboundary}.  Note that under this convention, changing $q$ by a positive multiple yields the same orientation, while the same set $\PP V \setminus \{-q = 0\}$ has $(-1)^r$ times the orientation of $\PP V \setminus \{q = 0\}$.  Lastly, given a chamber in $U(\RR)$ corresponding to a tope $P$, its image in $\overline U(\RR)$, which is the interior of the projective polytope $\overline P$, is oriented by choosing $q\in V^\vee$ to be any element in the interior of the inner dual cone $\sigma_P^\vee$ of the cone $\sigma_P$.
\end{rem}

Given a tope $P$ of $\cM$ and a choice of a chirotope $\chi$ on $\cM$, the projective polytope $\overline P$ in $\PP V$ is oriented as in the remark above, and let $\Omega_{\overline P}$ be the resulting canonical form.  Note that while both the tope $P$ and its negative $-P$ define the same projective polytope $\overline P = \overline{-P}$, we have $\Omega_{\overline P} = (-1)^r\Omega_{\overline{-P}}$.

To explain how $\Omega_{\overline P}$ relates to $\bOmega(P,\chi)$ of Definition~\ref{def:main}, we prepare with the following well-known result.
Let $X$ (resp.\ $\overline X$) be a log smooth compactification of the complexification $U(\CC)$ (resp.\ $\overline U(\CC)$) with a normal crossing boundary divisor $D = X \setminus U$ (resp. $\overline D = \overline X \setminus U$), and let $\Omega^\bullet_{\log}(X,D)$ denote the ring of global algebraic forms on $X$ with logarithmic singularities along $D$.

\begin{thm} \label{thm:OS}
With $\CC$-coefficients (i.e.\ $R = \CC$), the maps 
\begin{equation}\label{eq:Bri}
e_i \longmapsto \dlog \ell_i = \frac{d \ell_i}{\ell_i} \qquad \text{and} \qquad \partial e_{ij} \longmapsto \dlog \ell_i - \dlog \ell_j = \dlog(\ell_i/\ell_j)
\end{equation}
respectively induce isomorphisms 
\[
\OS^\bullet_\CC(\rM) \cong 
H^\bullet(U(\CC),\CC) \cong \Omega^\bullet_{\log}(X,D)\qquad \text{and} \qquad \oOS^\bullet_\CC(\rM) \cong   
H^\bullet(\overline U(\CC),\CC) \cong \Omega^\bullet_{\log}(\overline X,\overline D).
\]
\end{thm}

\begin{proof}
\cite{Bri,OS80} state the isomorphisms between the Orlik--Solomon algebras and the cohomology rings.  The isomorphisms to the ring of algebraic forms with log singularities follows from the fact that mixed Hodge structure of $H^k(U(\CC))$ (resp.\ $H^k(\overline U(\CC))$) is pure of Hodge type $(k,k)$ \cite{Sha93, Kim94}.
\end{proof}

\begin{thm}\label{thm:real}
Let $\chi$ be a chirotope on $\cM$, and $P$ a tope of $\cM$.
Under the assignment \eqref{eq:Bri} of Theorem~\ref{thm:OS}, the canonical form $\bOmega(P,\chi)$ is sent to the canonical form $\Omega_{\overline P}$ of the chamber closure $\overline P$ with the orientation as given in Remark~\ref{rem:orientP}.
\end{thm}

\begin{proof}
Note first that the facets of the polytope $\overline P$ correspond to the atoms $i$ of the matroid $\rM$ such that $i \in \partial P$ as defined above Corollary~\ref{cor:topeform}.  Now, the statement for $\bOmega(P,\chi)$ follows from Corollary~\ref{cor:topeform}, since the recursive property stated in the corollary matches the recursive property of $\Omega_{\overline P}$ in Definition~\ref{def:PG} of a positive geometry.
\end{proof}

\begin{rem}
For essential real hyperplane arrangements, Yoshinaga \cite{Yoshinaga} defines a \emph{chamber basis} for the Orlik--Solomon algebra $\oOS^\bullet_{\CC}(\rM)$.  We show in \cref{cor:Yos} below that Yoshinaga's basis agrees with that coming from canonical forms.
\end{rem}

\begin{rem}\label{rem:Fil}
Via \cref{thm:real}, the formula \eqref{eq:cantriang} gives a new formula for the canonical form $\Omega_{\overline P}$ of a projective polytope $\overline P$ in the following way.
For any triangulation $\{T_1, \dotsc, T_m\}$ of the polar dual polytope $\overline{P}^\circ$ of $\overline{P}$, one has that $\Omega_{\overline P} = \sum_{i=1}^m \pm \Omega_{\overline T_i'}$ with appropriate signs, where the $\overline T_i'$ are simplices constructed as modified polar duals of the $T_i$.  For details, we point to an illustrated example in Section~\ref{ssec:rk3eg}.

This formula can be used to recover the result of Filliman \cite{Fil92} that expressed the volume of a polytope $P$ in terms of volumes of simplices from a triangulation of the polar dual $P^\circ$.  We remark that the proof given in \cite{Fil92} does not generalize easily to oriented matroids; it involves analyzing how a certain quantity changes as the linear functional $q$ (which defines a realizable extension of the hyperplane arrangement) varies in $V^\vee$.  Analogous analysis for an arbitrary oriented matroid $\cM$ would take place in the extension space $\mathcal E(\cM)$, whose structure is known to be intricate. 
\end{rem}

\section{Basis for the Orlik--Solomon algebra}

We now produce a basis for $\oOS^\bullet(\rM)$ using canonical forms of topes of successive truncations of $\mathcal M$.
This depends on a choice of a sequence of general extensions.
In this section, the term ``hyperplane'' refers to a flat of corank 1 in a matroid, in contrast to its usage in the previous section for realizable matroids.

\medskip
Let $\cM \cup q$ be an oriented matroid extension.  For every tope $P$ of $\cM$, we have that either $(P,+)$ or $(P,-)$ or both is a tope of $\cM \cup q$.
The extension $\cM \cup q$ is called \emph{general} if no hyperplane of $\rM \cup q$ is of the form $H\cup q$ for $H$ a hyperplane in $\rM$.  General extensions exist by \cite[Chapter 7]{BLVSWZ99}.  

Let $\L$ denote the poset of signed covectors of $\cM$.  We say that a covector $X \in \L$ is \emph{contained in} a tope $P \in \T$ if $P \leq X$ in $\L$.

\begin{defn}
For an element $0\in E$, a tope $P$ of $\cM$ is a \emph{bounded tope} of the affine oriented matroid $(\cM,0)$ if every nontrivial covector $X$ of $\cM$ contained in $P$ satisfies $X(0) = +$; see \cite[Definition 4.5.1]{BLVSWZ99}.
For an extension $\tM = \cM \cup q$ of $\cM$, we say that a  tope $P$ of $\cM$ is \emph{bounded with respect to $q$} if $(P,+)$ is a bounded tope of the affine oriented matroid $(\cM\cup q, q)$.  Denote by $\T^0(\cM)$ and $\T^q(\cM)$ the set of topes bounded with respect to $0$ and with respect to $q$.
\end{defn}

Technically, the notation $\T^?$ is overloaded --- $\T^0$ is with respect to the element $0$ already in the ground set, whereas as $\T^q$ references the extended element $q$, but we trust that this will not cause confusion.

\begin{thm}\label{thm:basis}
Fix a chirotope $\chi$ on $\cM$, and a general extension $\cM \cup q$.  Then, we have that
\[
\{\bOmega(P,\chi) : P \in \T^q(\cM)\} \quad\text{is an $R$-basis of}\quad \oOS_R^{r-1}(\rM)
\]
for any commutative ring $R$.
\end{thm}

By \cref{prop:partialiso}, we have the following corollary.
\begin{cor}\label{cor:basis}
Fix $\chi$ on $\cM$, and a general extension $\cM \cup q$.  Then, we have that
\[
\{\tOmega(P,\chi) : P \in \T^q(\cM)\} \quad\text{is an $R$-basis of}\quad \OS_R^{r}(\rM)
\]
for any commutative ring $R$.
\end{cor}

We prove \cref{thm:basis} in \cref{prop:spans} and \cref{prop:indep} below.  We prepare with some notations.
For a basis $B$ of $\cM$, which is a basis in $\cM\cup q$, denote the fundamental signed circuit by $C_{B\cup q}$ where $C_{B\cup q}(q) = -$.
When $\cM$ is Boolean, for a fixed general extension $\cM\cup q$, note that $C_{B\cup q}|_B$ is the unique tope of $\cM$ bounded with respect to $q$, and $\bOmega(C_{B\cup q}|_B, \chi) = (-1)^{|C_{B\cup q}^-|-1}\chi(B)\partial e_B$.
Let us also denote for a basis $B$ of $\cM$
\[
\T_{B}^q(\cM) := \{P \text{ a tope of $\cM$} : P|_B = C_{B\cup q}|_B\}.
\]
Note that we have $P\in \T_B^q(\cM) \implies P \in \T^q(\cM)$ because every covector of $\cM\cup q$ is orthogonal to $C_{B\cup q}$.
We simply write $\T_B^q$ when the oriented matroid $\cM$ is understood.

\begin{prop}\label{prop:spans}
For any basis $B$ of $\cM$, we have an identity 
\[
(-1)^{|C_{B\cup q}^-|-1}\chi(B)\partial e_B = \sum_{P \in \T_B^q} \bOmega(P,\chi)
\]
of elements in $\oOS^{r-1}(\rM)$.
In particular, the set $\{\bOmega(P,\chi) : P \in \T^q(\cM)\}$ spans $\oOS^{r-1}(\rM)$.
\end{prop}

\begin{rem}\label{rem:simplex}
In the realizable case, the form  $(-1)^{|C_{B\cup q}^-|-1}\chi(B)\partial e_B$ has the following meaning.
Let $\cM$ be realized by a set of linear functionals $\{\ell_i \in V^\vee : i\in E\}$ as in the previous section.  For a basis $B$ of $\cM$, consider the arrangement of hyperplanes $\{\ell_i = 0\}_{i\in B}$, which divides $V\simeq \RR^r$ into $2^r$ chambers.  For $q\in V^\vee$ in general position, exactly one of the $2^r$ chambers is contained in $\{q >0\}$.  That is, exactly one of the $2^{r-1}$ chambers in $\PP^{r-1}$ is bounded with respect to $q$.  The simplex in $\PP^{r-1}$ which is the closure of this chamber has $(-1)^{|C_{B\cup q}^-|-1}\chi(B)\partial e_B$ as its canonical form.\end{rem}

\begin{proof}[Proof of \cref{prop:spans}]
We induct on the rank of $\cM$.  When $r = 1$, for any fixed $e\in E$, the left-hand-side (LHS) is $(-1)^{|C_{e\cup q}^-|-1}\chi(e)\cdot 1$, whereas the right-hand-side (RHS) is $P(e)\chi(e)$ where $P$ is the tope with $P(e) = C_{e\cup q}|_e = (-1)^{|C_{e\cup q}^-|-1}$.  For the inductive step with $r >1$, we show that taking $\operatorname{Res}_a$ on both sides gives an equality for all $a\in \mathfrak A(\rM)$, which yields the desired equality by Lemma~\ref{lem:Resinj}.
For a fixed $a\in \mathfrak A(\rM)$, the residue $\operatorname{Res}_a$ of RHS is
\[
\Res_a  \sum_{P \in \T_B^q} \bOmega(P,\chi) = \sum_{\substack{P \in \T_B^q\\ \partial P \ni a}} \bOmega(P/a, \chi/^P a).
\]
We now break into two cases:
\begin{enumerate}[label = (\roman*)]
\item If $a\cap B = \emptyset$, then $\operatorname{Res}_a$ of LHS is zero.
For a tope $P$, if $P_a$ is a covector then $P_a \circ -P$ (i.e.\ the same as $P$ except the $a$ part has flipped sign) is also, and moreover if $P \in \T_B^q$ then $P_a \circ -P \in \T_B^q$ also if $a\cap B = \emptyset$.  Since $- \chi/^P a = \chi/_{(P_a \circ -P)} a$, we find that every summation $\sum_{\substack{P \in \T_B^q\\ \partial P \ni a}}\bOmega(P/a, \chi/^P a)$ in the above is zero by pairwise cancellations.
\item If $a \cap B = \{i\}$, then $\operatorname{Res}_a$ of LHS is
$(-1)^{|C_{B\cup q}^-|-1}(\chi/i)(B\setminus i) \partial e_{B\setminus i}$.
Note that $C_{(B\setminus i)\cup q}$, the fundamental signed circuit of $B\setminus i$ in $\cM\cup q /a = \cM/a \cup q$, is $C_{B\cup q}|_{(B\setminus i) \cup q}$.
Hence, by \eqref{eq:Pa}, we have
\[
\{P/a : P\in \T_B^q(\cM) \text{ and } a\in \partial P\} = \T_{B\setminus i}^q(\cM/a).
\]
and moreover $P(i) (-1)^{|C_{B\cup q}^-|} = (-1)^{|C_{(B\setminus i) \cup q}^-|}$ since $P|_B= C_{B\cup q}|_B$.
Since $\chi/^Pa = P(i) (\chi/i)$ by definition, we thus compute $\sum_{\substack{P \in \T_B^q\\ \partial P \ni a}} \bOmega(P/a, \chi/^P a) = C_{B\cup q}(i) \cdot \sum_{P' \in \T_{B\setminus i}^q(\cM/a)} \bOmega(P',\chi/i)$, and the desired equality follows from the induction hypothesis. \qedhere
\end{enumerate}
\end{proof}

\begin{prop}\label{prop:indep}
Fix $\chi$ on $\cM$, and a general extension $\cM \cup q$.  Then
$\{\bOmega(P,\chi) : P \in \T^q(\cM)\}$ is $R$-linearly independent in $\oOS^{r-1}(\rM)$.
\end{prop}

\begin{proof}
We induct on the rank of $\cM$, where the $r = 1$ case is easily verified.
For $r>1$, suppose that $\sum_P c_P \bOmega(P,\chi) = 0$ for some collection of coefficients $c_P \in R$.
Suppose for a contradiction that the support set $\{P : c_P \neq 0\}$ is nonempty.  
Because the tope graph of $\cM$ is connected \cite[Proposition 4.2.3 and Lemma 4.2.2]{BLVSWZ99} and the support set is a proper subset of all topes, there exists $P$ in the support set and $a\in \mathfrak A(\rM)$ such that $a\in \partial P$ but $P/a \circ -P$ is not in the support set.  Taking the residue $\operatorname{Res}_a \sum_P c_P \bOmega(P,\chi)$, the $\bOmega(P/a,\chi/a)$ term has coefficient exactly $c_P$, which is zero by the induction hypothesis, contradicting $c_P \neq 0$.
\end{proof}

Let $\rM \cup q$ be a general extension of $\rM$.  We call the contraction $\rM^{(1)} := (\rM \cup q)/q$ a \emph{general truncation} of $\rM$.  Repeating this operation, we have general truncations $\rM^{(1)}, \rM^{(2)}, \ldots$ of ranks $r-1,r-2,\ldots$.  Note that the ground sets of $\rM$ and $\rM^{(k)}$ are identical.

\begin{prop} \cite[Exercise 3.5]{Dimca} \label{prop:trunc}
For any general truncation $\rM^{(k)}$ of $\rM$, the identification of ground sets $E(\rM) \cong E(\rM^{(k)})$ induces an $R$-module injection $\oOS^\bullet(\rM^{(k)}) \hookrightarrow \oOS^\bullet(\rM)$ which is an isomorphism in degrees $\leq r-k-1$.
\end{prop}

Let us now fix a general flag $\mathscr F$ of $\cM$, by which we mean a sequence $\cM^{(0)}, \cM^{(1)}, \dotsc, \cM^{(r-1)}$ of general truncations of $\cM$.  That is, $\cM^{(0)} = \cM$, and for all $k \geq 1$ we have $\cM^{(k)} = ( \cM^{(k-1)} \cup q_k)/q_k$ for a general extension $\cM^{(k-1)} \cup q_k$ of $\cM^{(k-1)}$ .
A tope $P \in \T$ is \emph{$k$-bounded} (with respect to $\mathscr F$) if it is a bounded tope with respect to $q_k$ in the truncation $\cM^{(k-1)}$.
Let $\T^{q_1,q_2,\ldots,q_k}$ denote the set of $k$-bounded topes of $\cM$.  
For a chirotope $\chi$, we let $\chi^{(k-1)}$ denote the corresponding chirotope on the truncation $\cM^{(k-1)}$ and let $\bOmega(P, \chi^{(k-1)})$ denote the canonical form of a tope $P \in \T(\cM^{(k-1)})$.
Via \cref{prop:trunc}, we may view $\bOmega(P, \chi^{(k-1)})$ as an element of $\oOS^{r-k}(\rM)$, and by applying Theorem~\ref{thm:basis} we conclude the following.

\begin{cor}\label{cor:trunc}
Fix $\chi$ on $\cM$, and a general flag $\mathscr F$ on $\cM$.  Then, we have that
\[
\{\bOmega(P,\chi^{(k-1)}) : P \in \T^{q_1,q_2,\ldots,q_k}(\cM)\} \quad\text{is an $R$-basis of}\quad \oOS_R^{r-k}(\rM)
\]
for any commutative ring $R$.
\end{cor}

We recover the following result due to Greene and Zaslavasky \cite{GZ} and Las Vergnas \cite{LV77}.

\begin{cor}
For $1 \leq k \leq r$, we have $\dim \oOS^{r-k}(\rM) = | \T^{q_1,q_2,\ldots,q_{k-1},q_k}(\cM)|$.
\end{cor}

\begin{cor}\label{cor:Yos} Suppose that $\cM$ is the oriented matroid associated to a hyperplane arrangement $\cA$.  Then the basis of \cref{cor:trunc} agrees with that of Yoshinaga \cite{Yoshinaga} up to a factor of $2\pi i$.
\end{cor}
\begin{proof}
It suffices to prove this for $k = 1$.  The identity of \cref{prop:spans} uniquely determines the basis of \cref{thm:basis}.  This identity is the same as the map named $\xi$ in \cite[Theorem 3.2]{Yoshinaga}.  The difference in a factor of $2\pi i $ arises from that the isomorphism \eqref{eq:Bri} is given instead by $e_i \mapsto \frac{1}{2\pi i }\dlog \ell_i$ in \cite{Yoshinaga}.
\end{proof}

\begin{rem}\label{rem:mult}
It would be interesting to give a simple description of the structure constants of the algebra $\oOS^\bullet(\rM)$ with respect to the basis of \cref{thm:basis}.  An alternating triple summation formula can be obtained by combining the following ingredients:
\begin{enumerate}
\item the formula \eqref{eq:cantriang} for $\rM^{(k)}$ expresses the canonical forms $\bOmega(P,\chi)$ in terms of the spanning set $\{\partial e_S \mid S \text{ independent}\}$ of $\oOS^\bullet(\rM)$;
\item we have the product formula 
$$\partial e_S \partial e_T = \sum_{i=0}^{\ell-1} (-1)^{i + \ell - 1} \partial e_{S \cup (T \setminus t_{\ell-i})}$$
in $\oOS^\bullet(\rM)$, where $S = \{s_1,\ldots,s_k\}$ and $T = \{t_1,\ldots, t_\ell\}$;
\item \cref{prop:spans} expands $\{\partial e_B \mid B \text{ basis}\}$ in terms of the canonical form basis of \cref{thm:basis}.
\end{enumerate}
\end{rem}

\section{Basis for Aomoto cohomology}

In this section, we take $R = \CC$, so that $\oOS^\bullet_{\CC}(\rM)$ is a graded finite-dimensional $\CC$-algebra.  We assume that a chirotope $\chi$ of $\cM$ has been fixed and write $\bOmega(P)$ for $\bOmega(P,\chi)$.
We fix an element $0\in E$ in the ground set, and work with the affine oriented matroid $(\cM, 0)$.  Recall from \eqref{eq:affine} (with $a = 0$) that we may view $\oOS^\bullet(\rM)$ as generated by the $\be_i := e_i - e_0$.
Without loss of generality, we assume that $\rM$ is simple.

\medskip
Let $\omega = \sum_{s\in E\setminus 0} \lambda_s \be_s \in \oOS^1(\rM)$ be a degree one element, where $\lambda_s \in \CC$ are complex numbers.  Consider the complex $(\oOS^\bullet,\omega)$
$$
0 \to \oOS^0(\rM) \stackrel{\omega}{\longrightarrow} \oOS^1(\rM)  \stackrel{\omega}{\longrightarrow} \oOS^2(\rM)  \stackrel{\omega}{\longrightarrow} \cdots  \stackrel{\omega}{\longrightarrow} \oOS^{r-1}(\rM) \stackrel{\omega}{\longrightarrow} 0,
$$
which is sometimes called the \emph{Aomoto complex}.
In the case of a projective complex hyperplane arrangement complement $\overline U(\CC)$, under a genericity assumption on $\omega$, the cohomology $H^*(\oOS^\bullet(\rM),\omega)$ of the Aomoto complex was shown to be equal to the twisted algebraic de Rham cohomology $H_{dR}^*(U(\CC),d + \omega \wedge)$, which is further equal to the topologically defined cohomology $H^*(U,\L_\omega)$ of a rank-one local system $\L_\omega$ when the $|\lambda_s|$ are sufficiently small \cite{ESV92}.  In this setting, the monodromy of the local system $\L_\omega$ around the hyperplane $H_s$ is given by $\exp(-2 \pi i \lambda_s)$.

For a generic $\omega \in \oOS^1(\rM)$, the cohomology $H^*(\oOS^\bullet,\omega)$ is in fact concentrated in degree $r-1$ \cite{SV91, Yuz95}.
Our goal is to construct a basis for this \emph{Aomoto cohomology} 
$$
H^{r-1}(\oOS^\bullet,\omega) := \oOS^{r-1}(\rM)/(\omega \wedge \oOS^{r-2}(\rM)).
$$
We prepare by noting the following numerology; see \cite{GZ, Yuz95}.  Let $\beta(\rM)$ denote the beta invariant of a matroid $\rM$.

\begin{prop}\label{thm:bounded}
We have $|\T^0(\cM)| = \beta(\rM)$, and $\dim H^{r-1}(\oOS^\bullet,\omega)  = \beta(\rM)$ for generic $\omega \in \oOS^1$.
\end{prop}

Let $V \subset \oOS^{r-1}(\rM)$ denote the subspace spanned by the canonical forms $\{\bOmega(P) \mid P \in \T^0\}$.  Our main theorem is the following.  In the realizable case, a similar result appeared in \cite[Theorem 3.1]{BY16}.  The decomposition $\oOS^{r-1}(\rM) = V\oplus W$ that we will define below corresponds to the decomposition denoted $R[bch(\mathcal A)] \oplus R[uch(\mathcal A)]$ for the realizable case in \emph{loc.~cit.}.

\begin{thm}\label{thm:Aomoto}
For generic $\omega \in \oOS^1(\rM)$, the subspace $V$ maps isomorphically to the cohomology $H^{r-1}(\oOS^\bullet,\omega)$.  Thus $\{\bOmega(P) \mid P \in \T^0\}$ form a basis of $H^{r-1}(\oOS^\bullet,\omega)$ over $\CC$.
\end{thm}

For $r = 1$, the statement is clear, so we assume that $r \geq 2$.
We shall prove \cref{thm:Aomoto} by a deletion-contraction induction on the rank $r$ and the cardinality of $E$.
Let us hence assume that \cref{thm:Aomoto} has been shown for oriented matroids that either have fewer elements or lower rank than $\cM$.

\def\starq{q}

Our proof will involve a direct sum decomposition of $\oOS^{r-1}(\rM)$ that depends on the following choice.
Choose a general extension $\tM = \cM \cup \starq$ of $\cM$ by $\starq$, and let $\tilde E =E \cup \starq$.  By letting $\starq$ be a general perturbation of $0$ (see for example \cite[Section 1.2]{San02}), we may assume that if a tope $X$ of $\cM$ is bounded with respect to $0$ then it is also bounded with respect to $\starq$.
As before, let $\T = \T(\cM)$ denote the set of topes of $\cM$, and let $\T^0(\cM) \subset \T$ (resp.\ $\T^\starq(\cM) \subset \T$) be the subset of topes that are bounded with respect to $0$ (resp.\ $\starq$).

\smallskip
We now define the direct sum decomposition of $\oOS^{r-1}(\rM)$.  Let 
$$
V = \bigoplus_{P \in \T^0} \CC \cdot \bOmega(P), \qquad W = \bigoplus_{P \in \T^\starq \setminus \T^0} \CC \cdot \bOmega(P).
$$
Since we have assumed that $\T^0 \subset \T^\starq$, by \cref{thm:basis}, we have 
\begin{equation}
\label{eq:mainsum}
\oOS^{r-1}(\rM) = \bigoplus_{P \in \T^\starq} \CC \cdot \bOmega(P)= V \oplus W.
\end{equation}
Let $\pi:\oOS^{r-1}(\rM) \to W$ denote the orthogonal projection to $W$ with respect to the direct sum decomposition \eqref{eq:mainsum}.  We shall show that $\pi(\omega \wedge \oOS^{r-2}(\rM)) = W$.
This implies \cref{thm:Aomoto} since it implies that $V$ spans $H^{r-1}(\oOS^\bullet,\omega)$, and the vector spaces $V$ and $H^{r-1}(\oOS^\bullet,\omega)$ have the same dimension (\cref{thm:bounded}).
Let $(\cM,\cM',\cM'')$ be a deletion-contraction triple with respect to $i \in E \setminus 0$.  We have the residue exact sequence \eqref{eq:oRes}
$$
0 \to \oOS^{r-1}(\rM')  \to \oOS^{r-1}(\rM) \stackrel{\Res_i}{\longrightarrow} \oOS^{r-2}(\rM'') \to 0.
$$
Define $V',V''$, $W', W''$ and $\pi', \pi''$ similarly for $\cM', \cM''$.   
Recall that $i$ is a said to be a facet of a tope $P$ if $X$ defined by $X(j)=P(j)$ for $j \neq i$ and $X(i) = 0$ is a covector.  We say that two topes $P$ and $Q$ \emph{share a facet $i$} if $P(j) = Q(j)$ for all $j \neq i$, and $P(i)Q(i) = -$.

\smallskip
Let $\omega' = \sum_{j \in E \setminus \{0,i\}} a_j \be_j$, which we view as both an element of $\oOS^1(\rM')$ and as an element of $\oOS^1(\rM'')$.

\begin{prop}\label{prop:omega'}
Suppose that $\omega'$ is generic.  Then we have $\pi(\omega' \wedge \oOS^{r-2}(\rM)) = W$.
\end{prop}
\begin{proof}
The statement is easy to check directly when $r = 2$, so we assume that $r > 2$.  If $i$ is a coloop of $\cM$, then the maps $\Res_i: \oOS^{r-1}(\rM) \to \oOS^{r-2}(\rM'')$ and $\Res_i: \oOS^{r-2}(\rM) \to \oOS^{r-3}(\rM'')$ are isomorphisms.  The statement $\pi(\omega' \wedge \oOS^{r-2}(\rM)) = W$ follows from the corresponding statement for $\rM''$.  Henceforth, we assume that $i$ is not a coloop.

Let
$$
K:=\ker(\Res_i|_W) \qquad \text{and} \qquad R:=\Res_i W.
$$
We will show that 
$$
K \subseteq \pi(\omega' \wedge \oOS^{r-2}) \qquad \text{and} \qquad  R \subseteq \Res_i(\pi(\omega' \wedge \oOS^{r-2})).
$$
Since $ \pi(\omega' \wedge \oOS^{r-2}) \subseteq W$, together these statements imply that $ \pi(\omega' \wedge \oOS^{r-2})  = W$.

Let $B''$ denote the basis of \cref{thm:basis} applied to the oriented matroid $\cM'' = \cM/i$ with general extension $\tM/i$.
The map $\Res_i$ sends each element of $\{\bOmega(P) \mid P \in \T^\starq\}$ either to 0 or to an element of $B''$, up to sign.  It follows that $K$ is spanned by the following two kinds of classes: (i) $\bOmega(P)$ for $P \in \T^\starq - \T^0$, where $i$ is not a facet of $P$, and (ii) $\bOmega(P) + \bOmega(Q)$ where $P,Q \in \T^\starq - \T^0$ are topes that share the facet $i$.

For (i), we have $\bOmega(P) \in W'$, so by the induction hypothesis applied to $\cM'$, there exists $x \in \oOS^{r-2}(\rM')$ such that
$$
\pi'(\omega' \wedge x) = \bOmega(P).
$$
But $\ker(\pi') = V' \subset V = \ker(\pi)$ so it follows that $\pi(\omega' \wedge x) = \bOmega(P)$ and thus $\bOmega(P) \in \pi(\omega' \wedge \oOS^{r-2} )$.  For (ii), the same argument shows that $\bOmega(P)+\bOmega(Q) \in \pi(\oOS^{r-2} \wedge \omega')$.  We conclude that $K \subseteq \pi(\oOS^{r-2} \wedge \omega')$.

Now we consider $R$.  The space $R$ is spanned by the classes $\Res_i \bOmega(P)$ where $P \in \T^\starq -\T^0$ has $i$ as a facet.  Let $P$ be such a tope, and let $Q$ be the tope such that $P,Q$ share the facet $i$.  If $Q \in \T^0$, then $\Res_i \bOmega(P) = \bOmega(P'') \in V''$, where $P'' \in (\T'')^0$.  In this case, $ \bOmega(P) + \bOmega(Q) \in W'$ so by the same argument as above, there exists $x \in \oOS^{r-2}(\rM')$ such that $\pi(x \wedge \omega') = \bOmega(P)$.
Write $\tV'' \subset V''$ for the subspace spanned by $\Res_i\bOmega(P)$ for $P \in \T^\starq -\T^0$ with $i$ as a facet and $Q \in \T^0$ on the other side.  We have shown that $\tV'' \subset  \Res_i( \pi(\omega' \wedge\oOS^{r-2}))$.

Suppose that $P \in \T^\starq - \T^0$ and $P,Q$ share the facet $i$, with $Q \notin \T^0$.  Then $\Res_i \bOmega(P) = \bOmega(P'')$ where $P'' \in (\T'')^\starq - (\T'')^0$.  By the inductive hypothesis applied to $\cM''$, there exists $y \in \oOS^{r-3}(\rM'')$ such that
$$
\pi''(\omega' \wedge y) = \bOmega(P''), \qquad \text{or} \qquad \omega' \wedge y= \bOmega(P'') \mod V''.
$$
We have $y \wedge e_i \in \oOS^{r-2}(\rM)$ satisfying 
$$
\Res_{i}(\omega' \wedge y \wedge e_i) = \omega' \wedge y = \bOmega(P'') \mod V''.
$$
Expanding $\omega' \wedge y \wedge e_i$ into the basis of canonical forms given by \cref{thm:basis}, we see that this implies
$$
\Res_{i}(\pi(\omega' \wedge y \wedge e_i)) = \bOmega(P'') \mod \tV''.
$$
Since we have shown that $\tV'' \subset \Res_i(\pi(\omega' \wedge \oOS^{r-2}))$, this shows that $\bOmega(P'') \in \Res_i(\pi(\omega' \wedge \oOS^{r-2}))$.
We conclude that $R  \subseteq \Res_i(\pi(\omega' \wedge \oOS^{r-2}))$ and it follows that $\pi(\omega' \wedge \oOS^{r-2}) = W$.
 \end{proof}
 
 \begin{proof}[Proof of \cref{thm:Aomoto}]
 Let $\omega = \omega' + a_i \be_i$ be an element of $\oOS^1(\rM)$.  Consider the linear map 
 $$
 f_{a_i}: \oOS^{r-2} \to W, \qquad z \longmapsto \pi(\omega \wedge z) = \pi(\omega' \wedge z) + a_i \pi(\be_i \wedge z ).
 $$
 Thus $f_{a_i} = f + a_i f'$, where $f$ is the linear map $z \mapsto \pi(\omega' \wedge z) $ discussed in \cref{prop:omega'}.  By \cref{prop:omega'}, $f$ has full rank.  It follows that for generic values of $a_i$ (including for all values of $a_i$ sufficiently close to $0$) we have that $f_{a_i}$ has full rank.  Thus for a generic $\omega \in \oOS^1(\rM)$, we have $\pi(\omega \wedge \oOS^{r-2} ) = W$.  Since $V$ has dimension equal to $|\T^0|$, the theorem then follows from \cref{thm:bounded}.  
 \end{proof}

\section{Examples}\label{sec:examples}

\subsection{Rank two}
We work with $R = \CC$.  Let $\bA$ be the projective hyperplane arrangement consisting of $n$ points $z_0 = \infty, z_1,z_2,\ldots, z_n$ on $\PP^1(\RR)$, arranged in order.  The corresponding oriented matroid $\cM$ has rank $r = 2$, with ground set $E = \{0,1,\ldots,n\}$.  We have
$$
 \oOS^0 = \CC \cdot 1, \qquad  \oOS^1 = \bigoplus_{i=1}^n  \CC\cdot \be_i, \qquad\text{where $\be_i := e_i - e_0$ as before.}
$$
The topes of $\cM$ can be identified with the chambers of $\bA$, which are the $(n+1)$ intervals:
$$
(-\infty,z_1), (z_1,z_2), (z_2,z_3),\ldots, (z_n,\infty).
$$
By reorienting $\cM$ if necessary, we may assume that they correspond to the topes 
$$
P_0 = (+,-,-,\ldots,-), P_1 = (+,+,-,\ldots,-), P_2 = (+,+,+,-,\ldots,-),\ldots, P_n = (+,+,\ldots,+)
$$
all of which satisfy $P(0) = +$.  In this case, we may choose the chirotope $\chi$ to be 
$$
\chi(i,j) = \begin{cases} + & \mbox{if $i < j$,} \\
- & \mbox{if $i > j$.}
\end{cases}
$$
A triangulation of $P_i$ is given by the basis $\{i,i+1\}$.  Formula \eqref{eq:cantriang} gives
$$
\bOmega(P_i,\chi) = (-1) \chi(i,i+1) (e_i - e_{i+1}) = e_{i+1} - e_i.
$$
The canonical forms for $P_0,P_1,\ldots,P_n$ are
$$
e_1 - e_0, e_2 - e_1, e_3- e_2,\ldots,e_0 - e_n,
$$
respectively.  If the general extension $q$ is made such that the corresponding hyperplane (a point in this case) in $\PP^1(\RR)$ lies in  the chamber $(z_n,\infty)$,  then $\T^q$ consists of the remaining $n$ topes, which form a basis of $\oOS^1$.  Let $i < j < n$ and consider the basis $B = \{i,j\}$.  Then $\T^q_B$ consists of the topes $P_i,P_{i+1},\ldots,P_{j-1}$, and \cref{prop:spans} reduces to 
$$
(+1)(-1)(e_i - e_j) = \sum_{a=i}^{j-1} \bOmega(P_a,\chi).
$$
The set $\T^0$ consists of the $n-1$ bounded topes:
$$
(z_1,z_2), (z_2,z_3),\ldots, (z_{n-1},z_n).
$$
Now let $\omega = \sum_{i=1}^n a_i \be_i = - (\sum_{i=1}^n a_i)e_0 + \sum_{i=1}^n a_i e_i$.  Using the relation 
$$
\omega = a_{n}(e_n - e_{n-1}) + (a_n + a_{n-1})(e_{n-1} - e_{n-2}) + \cdots + (a_n+ a_{n-1} + \cdots + a_1)(e_1- e_0)
$$
we verify \cref{thm:Aomoto}.  The genericity condition on $\omega$ for \cref{thm:Aomoto} to hold is $ (a_n+ a_{n-1} + \cdots + a_1) \neq 0$.

\subsection{Rank three}\label{ssec:rk3eg}

We give an explicit rank 3 realizable example, illustrating the formula pointed out in Remark~\ref{rem:Fil}.  Let the columns of the matrix
\[
\begin{bmatrix}
1 & 0 & -1 & 0 & -1 \\
0 & 1 & 0 & -1 & -1 \\
1 & 1 & 1 & 1 & 1 
\end{bmatrix},
\]
considered as linear functionals on $V = \RR^3$, realize an oriented matroid $\cM$ of rank $3$ on the ground set $E = \{1, 2, \dotsc, 5\}$.  Let $q$ be an additional linear functional given by the column vector $(0,0,1)^T$, and $H_q = \{q = 0\}$ the corresponding hyperplane.
These are illustrated below.

\begin{figure}[h!]
\centering
\includegraphics[height = 30mm]{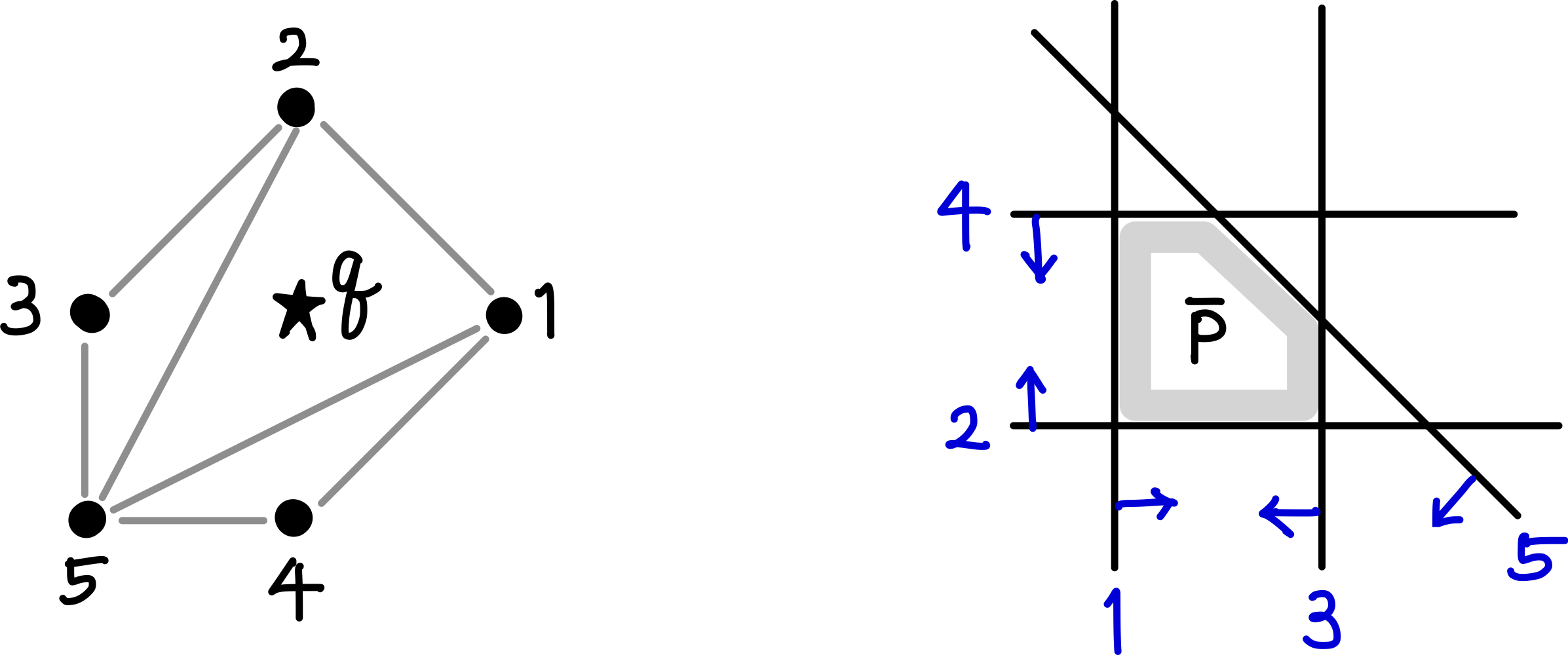}
\caption{Left: the affine diagram of the arrangement of points in $\PP(V^\vee)$.  Right: the corresponding hyperplane arrangement in $\RR^2 \cong \PP(V)\setminus \PP H_q$.}
\label{fig:rk3eg}
\end{figure}

Let $P$ be the tope $(+,\dotsc, +)$, which corresponds to be region $\overline P$ in gray on the right figure.  
With respect to the triangulation as drawn on the affine diagram on the left, its canonical form is
\begin{align*}
\Omega_{\overline P} &= \chi(125) \partial e_{125} + \chi(235)\partial e_{235} + \chi(145)\partial e_{145}\\
&= \Omega_{T_{125}} - \Omega_{T_{145}} - \Omega_{T_{235}}
\end{align*}
where $T_{ijk}$ denotes the bounded triangle in the right figure of Figure~\ref{fig:rk3eg} bounded by the lines labelled $i,j,k$.  Note that the signs of $\Omega_{T_{ijk}}$ agree with the sign $(-1)^{|C_{B\cup q}^-|-1}$ per Remark~\ref{rem:simplex}.

\small
\bibliography{EL_CanonFormOrientMat_refs}
\bibliographystyle{alpha}

\end{document}